\documentclass[12pt]{article}

\usepackage{amsfonts}
\usepackage{amssymb}

\usepackage{amsmath,amssymb}
\usepackage{epsfig,afterpage}
\usepackage{graphicx}
\usepackage{amsmath}
\usepackage{amsthm}
\usepackage{amscd}
\usepackage{amsbsy}
\usepackage{epsfig}
\usepackage[dvips]{psfrag}
\usepackage{makeidx}
\usepackage[all]{xy}
\usepackage{latexsym}
\usepackage{tikz}
\usetikzlibrary{matrix,arrows}
\makeindex
\usepackage[latin1]{inputenc}

\newtheorem{theorem}[subsection]{Theorem}

\newtheorem{remark}[subsection]{Remark}

\newtheorem{proposition}[subsection]{Proposition}
\newcommand{\R}{\mathbb{R}}

{ \small \normalsize}

\newfont{\bms}{msbm10 scaled 1200}
\newfont{\bmt}{msbm10 scaled 2500}

\newfont{\bmtt}{msbm10 scaled 1700}

\title{Examples of infinitesimal non-trivial accumulation of secants in dimension three}
\author{
        Andr\'e Belotto \\
        University of Toronto\\
        andre.belottodasilva@utoronto.ca
}
\date{10/2014}

\begin{document}
\maketitle
\section*{Abstract}
We present new examples of accumulation of secants for orbits (of a real analytic three dimensional vector fields) having the origin as only $\omega$-limit point. These new examples have the structure of a proper algebraic variety of $\mathbb{S}^2$ intersected with a cone. In particular, we present explicit examples of accumulation of secants sets which are not in the list of possibilities of the classical Poincaré-Bendixson Theorem.\\
\\
%
{\bf Keywords} Accumulation of secants, Ordinary Differential Equations, Real analytic geometry.
\def\contentsname{Contents}
\tableofcontents 
\thispagestyle{empty}
\newpage
\section{Introduction}
Let $X$ be a real analytic vector-field defined in a neighborhood of the origin of $\mathbb{R}^3$ and assume that the origin is a singularity of $X$. Consider a regular orbit $\gamma(t)$ of $X$ such that the origin is the only $\omega$-limit point of $\gamma(t)$, i.e $lim_{t\rightarrow \infty} \gamma(t) = (0,0,0)$. A classical problem is to study the qualitative behavior of the trajectory $\gamma(t)$ (see, for example, \cite{M1,M2,M3}). A modern approach to this problem is to ask for a description of the \textit{accumulation of secants} of $\gamma(t)$.\\
\\
More precisely, let the curve $\frac{\gamma(t)}{\|\gamma(t)\|} \subset \mathbb{S}^{2}$ be the \textit{secants} curve of $\gamma(t)$. The \textit{accumulation of secants} is the set: 
$$
Sec(\gamma(t)) = \bigcap_{s} \overline{\{ \gamma(t) / \|\gamma(t)\| ; t \geq s  \}} \subset \mathbb{S}^{2}
$$
In the literature we can find descriptions of accumulation os secants for some classes of three dimensional vector-fields. For example:
\begin{itemize}
\item In $\cite{Fel}$ the authors provide a description of the accumulation of secants $Sec(\gamma(t))$ under the hypotheses that the origin is a \textit{generic absolutely isolated singularity} (see definition in the same article);
\item In \cite{Fel1,Fel2,Fel3,Fel4}, the authors study properties of orbits under some additional oscillating hypotheses (as, the so-called, \textit{sub-analytically non-oscillating} property).
\end{itemize}
But, the problem of giving a description of the accumulation on secants $Sec(\gamma(t))$ for general three dimensional vector-fields is widely open. In one hand, it is still an open question whether there exists an accumulation of secants set $Sec(\gamma(t))$ which is dense on the sphere. At another hand, all examples in the literature are contained in the list of possibilities given by the Poincaré-Bendixson Theorem.\\
\\
In this work, we prove that a connected component of a proper algebraic variety of $\mathbb{S}^2$ intersected with a cone may be realized as an accumulation of secants set. In particular, we present a wide class of examples of accumulations of secants sets which are not in the list of possibilities of the Poincaré-Bendixson Theorem (see figure ($\ref{fig:1}$) for an example). Furthermore, the vector-fields that give rise to these new examples are explicit and polynomial.\\
\\
In order to be precise, consider a non-zero real-polynomial $h \in \mathbb{R}[x,y]$ and let $\Gamma$ be the semi-algebraic set $V(h) \cap B_1(0)$, where $B_1(0)$ is the closed ball with radius one and center at the origin and $V(h) = \{p \in \mathbb{R}^2; h(p)=0\}$ is the algebraic variety given by the zeros of the polynomial $h$. We say that the polynomial $h$ is \textit{adapted} if:
\begin{itemize}
\item The set $\Gamma$ is connected and non-empty;
\item The algebraic variety $V(h,\frac{\partial}{\partial x} h,\frac{\partial}{\partial y} h) = V(h,h_x,h_y)$ is equal to a finite set of points; 
\item The set $V(h) \cap \partial B_1(0)$ is a finite set of points such that $<\nabla^{\bot} h(p),p> \neq 0$, where $\partial B_1(0)$ stands for the border of $B_1(0)$.
\end{itemize}
Now, consider the following chart of $\mathbb{S}^2$:
$$
\begin{array}{cccl}
\alpha: & \mathbb{R}^2 & \longrightarrow & \mathbb{S}^2 \\
 & (x,y) & \mapsto & \left(\frac{x}{\sqrt{x^2+y^2+1}},\frac{y}{\sqrt{x^2+y^2+1}},\frac{1}{\sqrt{x^2+y^2+1}}\right)
\end{array}
$$
With this notation, the main result of this work can be precisely enunciated:
\begin{theorem}
Let $h \in \mathbb{R}[x,y]$ be an adapted polynomial. There exists a real three dimensional polynomial vector-field $X$ and a regular orbit $\gamma(t)$ of $X$ such that:
\begin{itemize}
\item[i]) The $\omega$-limit of the orbit $\gamma(t)$ is an isolated singularity of $X$;
\item[ii]) The accumulation of secants set $Sec(\gamma(t))$ is equal to $\alpha(\Gamma)$.
\end{itemize}
\label{th:main}
\end{theorem}
\begin{flushleft}
\textbf{Example 1:} \textit{A first non-trivial example can be given by the simple function:}
\end{flushleft}
$$
h(x,y) = (x^2-\frac{1}{4})(y^3-\frac{1}{4}y)
$$
\textit{Notice that the set $\Gamma = V(h) \cap B_1(0)$ is not in the list of possibilities given by the Poincaré-Bendixson Theorem (see figure \ref{fig:1}). Moreover, since $h$ is adapted, Theorem \ref{th:main} guarantees that $\alpha(\Gamma)$ is an accumulation of secants for some algebraic vector-field $X$.}\\
\\
\begin{figure}
  \centering
  \includegraphics[scale=0.4]{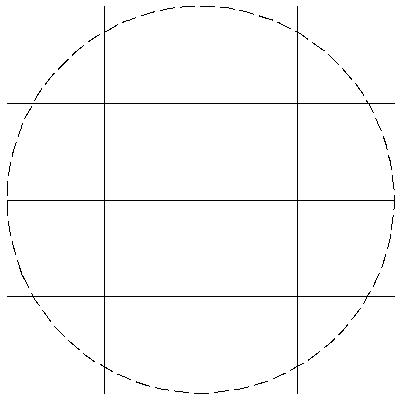}
  \caption{\small{The set $\Gamma = V(h) \cap B_1(0)$ of example 1.}}
  \label{fig:1}
 \end{figure}
The proof of Theorem \ref{th:main} is a direct consequence of Theorem $\ref{th:IS}$ and Proposition $\ref{prop:IS}$ below. Nevertheless, it is worth remarking that Propositions $\ref{prop:Bel}$ and $\ref{prop:glo}$ below are enough to prove property $[ii]$ of the Theorem.\\
\\
To conclude, we would like to remark that the ideas behind this work may lead to a stronger result. We believe that it should be possible to prove that any proper and connected semi-algebraic variety of $\mathbb{S}^2$ (maybe satisfying some generic condition) may be realized as an accumulation of secants set for some algebraic vector-field $X$.
\begin{remark}
In the rest of the manuscript we study the convergence of the orbit $\gamma(t)$ after the projective blowing-up of the origin:
$$
\tau : \mathbb{R} \times \mathbb{P}^2 \longrightarrow \mathbb{R}^3
$$
instead of the secants of $\gamma(t)$. We remark that this treatment leads to the desired result because $\mathbb{S}^2$ is a double cover of $\mathbb{P}^2$.
\end{remark}

\section*{Acknowledgments}
I would like to thank my thesis advisor, Professor Daniel Panazzolo, for the useful discussions, suggestions and revision of the manuscript. I would also like to thank Professor Felipe Cano for bringing the problem to my attention and Professor Patrick Speisseger for a very useful discussion. Finally, I would like to express my gratitude to the anonymous reviewer for all the suggestions and for indicating a shorter way of proving Proposition $\ref{prop:main}$ below. This work was supported by the Université de Haute-Alsace.
\section{The $\Gamma$-convergence}
Consider a real polynomial $H \in \mathbb{R}[x,y,z]$ and a compact and connected sub-set $\Gamma$ of $\mathbb{R}^2$. Let $L_{z_0}$ be the variety $V(H,z-z_0)$ for any constant $z_0$ in $\mathbb{R}$. We say that the polynomial $H$ is \textit{$\Gamma$-convergent} if the varieties $L_{z_0}$, with positive constants $z_0$, converge (in the Hausdorff topology) to $\Gamma$ when $z_0$ converges to zero, i.e. $L_{z_0} \stackrel{z_0\rightarrow 0^+}{\longrightarrow} \Gamma \times \{0\}$.
\begin{remark}
The fact that a polynomial $H$ is $\Gamma$-convergent does not imply that the variety $L_{0}$ is equal to $\Gamma$.
\end{remark}
\begin{flushleft}
Furthermore, we say that $H$ is:
\end{flushleft}
\begin{itemize}
\item \textit{connected} $\Gamma$-convergent if there exists a positive real number $\epsilon$ such that $L_{z}$ is a connected variety for $0<z<\epsilon$;
\item \textit{smooth} $\Gamma$-convergent if there exists a positive real number $\epsilon$ such that $L_{z}$ is a smooth variety for $0<z<\epsilon$.
\end{itemize}
\begin{remark}
If $H$ is connected $\Gamma$-convergent then, for a small enough positive real number $\epsilon$, the semi-analytic variety 
$V(H) \cap \{0<z<\epsilon\}$ is homeomorphic to a topological cylinder.
\end{remark}
If $\Gamma$ is equal to a semi-analytic variety $V(h) \cap \{x^2+y^2 \leq 1\}$, where $h$ is an adapted polynomial (see definition in the introduction), then Propositions $[2.1-2.2]$ of $\cite{Bel}$ guarantees the existence of a polynomial $H$ that is smooth connected $\Gamma$-convergent. More precisely, in $\cite{Bel}$ it is consider the family of polynomials:
$$
H_{\alpha}(x,y,z) = h(x,y)^2 + g(x,y,z)  + \alpha z^4
$$
where the parameter $\alpha$ takes values in the interval $[0,1]$ and the polynomial $g(x,y,z)$ has the form $z(\widetilde{g}(x,y,z) + \sum^N_{i=1} \binom{N}{i} x^{2i}y^{2(N-i)} )$ for some $N \in \mathbb{N}$, where $\widetilde{g}(x,y,z)$ is a polynomial with degree strictly smaller than $2N$. We remark that the explicit expression of the polynomial $g(x,y,z)$ can be found in $\cite{Bel}$, but is not important for this work. Finally, Propositions $[2.1-2.2]$ of $\cite{Bel}$ are sufficient to prove the following:
\begin{proposition}
If $\Gamma$ is equal to a semi-analytic variety $V(h) \cap \{x^2+y^2 \leq 1\}$, where $h$ is an adapted polynomial, then there exists a polynomial $H$ smooth connected $\Gamma$-convergent. Furthermore, such a polynomial can be chosen as the polynomials $H_{\alpha}(x,y,z)$ for almost all $\alpha$ fixed.
\label{prop:Bel}
\end{proposition}
In what follows we work with \textit{any} polynomial $H(x,y,z)$ that is smooth connected $\Gamma$-convergent.
%
%
%
\section{Construction in a chart}
We say that a real three dimension vector-field $X$ is \textit{$\Gamma$-convergent} if the variety $\{z=0\}$ is invariant by $X$, i.e. $X$ is tangent to $\{z=0\}$, and there exists an orbit $\gamma(t)$ of $X$ contained in $\{z>0\}$ such that the $\omega$-limit of $\gamma(t)$ is equal to $\Gamma \times \{0\}$.\\
\\
In this section, we prove that the existence of a polynomial smooth connected $\Gamma$-convergent implies the existence of a vector-field $\Gamma$-convergent. More precisely, let $H$ be smooth connected $\Gamma$-convergent and consider:
\begin{itemize}
\item The $H$-\textit{horizontal} vector-field $\mathcal{H}(H) = -H_y \frac{\partial}{\partial x} +  H_x \frac{\partial}{\partial y}$;
\item The $H$-\textit{vertical} vector-field $\mathcal{V}(H) = \mathcal{V}(H) : = H_x H_z \frac{\partial}{\partial x} +  H_y H_z \frac{\partial}{\partial y} -(H_x^2 + H^2_y) \frac{\partial}{\partial z}$;
\item The $H$-\textit{perturbed} vector-field $\mathcal{P}(H) : = H \frac{\partial}{\partial z}$.
\end{itemize}
Notice that, for a small enough $\epsilon>0$:
\begin{itemize}
\item The $H$-horizontal vector-field $\mathcal{H}(H)$ is the Hamiltonian of the function $H$ in respect to the variables $(x,y)$;
\item The $H$-horizontal vector-field $\mathcal{H}(H)$ and the $H$-vertical vector-field $\mathcal{V}(H)$ forms a base of $T_pV(H)$ for all points $p$ in the intersection of $V(H)$ and $\{0 <z<\epsilon\}$.
\item The $H$-\textit{perturbed} vector-field $\mathcal{P}(H)$ is identically zero in $V(H)$.
\end{itemize}
Now, consider the polynomial family of vector-fields:
$$
X_{\beta_1,\beta_2}(H) = \mathcal{H}(H) + z^{\beta_1}\mathcal{V}(H) + z^{\beta_2}\mathcal{P}(H)
$$
for constants $\beta_1$ and $\beta_2$ in $\mathbb{N} \setminus \{0\}$. The main result of this section can now be formulated precisely:
\begin{proposition}
Suppose that there exists a polynomial $H$ smooth connected $\Gamma$-convergent. Then, for any big enough $\beta_1$ and any positive $\beta_2$, the vector-field $X_{\beta_1,\beta_2}(H)$ constructed above is $\Gamma$-convergent.
\label{prop:main}
\end{proposition}
\begin{proof}
Fix a constant $\beta_2>0$ and let $X_{\beta_1}$ denote the vector-field $X_{\beta_1,\beta_2}(H)$ with the fixed $\beta_2$. By construction, it is clear that the vector-field $X_{\beta_1}$ is tangent to the variety $V(H)$ and to the plane $\{z=0\}$, which implies that the semi-algebraic set $V(H) \cap \{z \geq 0\}$ is invariant by the flux of $X_{\beta_1}$.\\
\\
We claim that for $\beta_1$ big enough, there exists $\epsilon>0$ such that all orbits of $X_{\beta_1}$ passing by $V(H) \cap \{0<z<\epsilon\}$ have $\omega$-limit equal to $\Gamma \times \{0\}$. Indeed, consider $\epsilon> 0$ sufficiently small, so that we can define the analytic diffeomorphism:
$$
\pi: \mathbb{S}^1 \times \{ 0  < r < \epsilon\} \to V(H) \cap \{0  < z < \epsilon\}
$$
where $\pi(\theta,r) = (x(\theta,r),y(\theta,r),r)$. We remark that one can prove the existence of such a map using the flux of the horizontal vector-field $\mathcal{H}(H)$.

\begin{remark}
 [This remark is supposed to be removed]. Since the horizontal vector-field $Y=\mathcal{H}(H)$ is analytic (even algebraic), the regular orbits of this vector-field are analytic functions and the dependence of the initial condition is also analytic. So, let $\Sigma$ be an analytic transverse section of $Y$ which is contained in $V(H) \cap \{0  < z < \epsilon\}$. Then the return map is well-defined and the function $T : \Sigma \to \mathbf{R}$ which gives the \textbf{time} of return is an analytic function which is always positive. We will use the analytic vector field $\frac{Y}{T}$, where we abuse notation to define $T$ over all $V(H) \cap \{0  < z < \epsilon\}$ by letting it be constant over the orbits of $Y$. So, we can consider the orbits $\phi: \mathcal{R} \times \Sigma \to  V(H) \cap \{0  < z < \epsilon\}$ where $\phi(t,p)$ is the orbit of the vector-field $\frac{Y}{T}$ with initial condition $p$ and time $t$. Since $\Sigma$ is clearly biholomorphic with $(0,\epsilon)$, this gives rise to the necessary analytic map.\\
\\
Furthermore, I believe the map can not be algebraic because we are using polar coordinates.
\end{remark}
We work with the pull-back of $X_{\beta_1}$ under this map:
\begin{align*}
\pi^{\ast}\mathcal{H}(H) & = A(\theta,r) \frac{\partial}{\partial \theta}\\
\pi^{\ast}\mathcal{V}(H) & = B(\theta,r) \frac{\partial}{\partial \theta} - C(\theta,r) \frac{\partial}{\partial r} \\
\pi^{\ast}X_{\beta_1} & = [A(\theta,r) + r^{\beta_1} B(\theta,r) ] \frac{\partial}{\partial \theta}  -  r^{\beta_1} C(\theta,r) \frac{\partial}{\partial r}
\end{align*}
where the functions $A$, $B$ and $C$ are analytic functions over the cylinder $\mathbb{S}^1 \times \{0  < r < \epsilon\}$.\\
\\
Now, we claim that the function $A$ and $C$ are strictly positive functions. Indeed, apart from taking a smaller $\epsilon$, we notice that the vector-field  $\mathcal{H}(H)$ has no singularities since $H$ is $\Gamma$-convergent (thus, the intersection of the variety $V(H,H_x,H_y)$ with the set $\{0 <z<\epsilon\}$ is empty). This implies that $A$ is of a fixed sign for $r>0$ and, apart from changing the sign of $\theta$, we can assume that it is positive.\\
\\
Furthermore, notice that the function $L:=z$ is a strict negative Lyapunov function of $\mathcal{V}(H)$ over the semi-analytic set $V(H) \cap \{0<z<\epsilon\}$. Indeed, for a point $p$ in $V(H) \cap \{0<z<\epsilon\}$:
$$
[\mathcal{V}(H)](L)=[-z^{\beta_1}(H_x^2 + H^2_y)] < 0
$$
This implies that $-C = \pi*[\mathcal{V}(H)](L) < 0$ and $C$ is strictly positive function. We remark that this already implies that all orbits of $X_{\beta_1}$ passing by  $V(H) \cap \{0<z<\epsilon\}$ have $\omega$-limit \textit{contained} in $\Gamma \times \{0\}$\\
\\
To be able to conclude, we just need to show that $  \frac{r d \theta}{ d r}> 1$ for all $\theta$ in $\mathbb{S}^1$ and all $r$ sufficiently small. This would clearly imply that the vector-field $\pi^{\ast}X_{\beta_1}$ is spiraling, i.e its orbits have the hole $\mathbb{S}^1 \times \{0\}$ as a $\omega$-limit. Finally, this would imply that all orbits of $X_{\beta_1}$ passing by $V(H) \cap \{0<z<\epsilon\}$ have $\omega$-limit equal to $\Gamma \times \{0\}$.\\
\\
So, we just have to prove that the following inequality holds:
\begin{equation}
\label{eq:2}
\frac{A}{C r^{\beta_1-1}} + \frac{r B}{C} > 1
\end{equation}
To be able to do so, we use Lojasiewicz's inequality for semi-analytic sets (see Theorem 6.4 of \cite{BM}). We actually use a slightly stronger result which can be easily deduced from the proof given in Theorem 6.4 of \cite{BM} (for the excat same remark for quasi-algebraic category, see Proposition $2.11$ of \cite{Co} and remark (a) below the same Proposition):
\begin{proposition}
[] Let $K$ be a semi-nalytic set of $\mathbb{R}^n$, $f$ and $g$ two continuous and semi-analytic functions over $K$ such that:
\begin{itemize}
 \item[i]) The set $\{ x \in K ; \|g(x)\| > \delta \}$ is compact for all $\delta>0$;
 \item[ii]) The zeros of $f$ are also zeros of $g$.
\end{itemize}
Then, there exists strictly positive constants $M$ and $\alpha$ such that $\|f(x)\| \geq M \|g(x)\|^{\alpha}$ for all $x$ in $K$. 
\label{prop:LI}
\end{proposition}
The idea is to apply this result for the semi-algebraic set $K = \mathbb{S}^1 \times \{0  < r < \epsilon\}$ and for the function $g =r$. We notice that these choices of $K$ and $g$ clearly satisfies hypothesis $[i]$ of Proposition \ref{prop:LI}. So, we can directly apply Proposition \ref{prop:LI} to the functions $A$, $C$ and $\frac{1}{C}$ in order to obtain positive constants $M$, $a$, $c_1$ and $c_2$ such that: $A > M r^{a}$; $C > M r^{c_1}$ and $\frac{1}{C} > M r^{c_2}$.\\
\\
We notice that one can not apply Proposition \ref{prop:LI} directly to $B$ since the function $B$ may have zeros on $K$. Nevertheless, since $B$ is analytic on $K$, let us consider the two semi-algebraic sets $K_1 = K \cap \{ B > -1 \}$ and $K_2 = K \cap \{ B \leq -1 \}$, where $K_2$ is relatively compact. In this situation, we can apply Proposition \ref{prop:LI} to the function $\frac{1}{B}$ over $K_2$ in order to obtain a positive constant $b$ such that, without loss of generality: $- \frac{1}{B}> M r^b$, which implies that $B> - M r^{-b}$.\\
\\
So, for a point $p$ in $K_1$, we obtain:
\begin{align*}
\frac{A}{C r^{\beta_1-1}} + \frac{r B}{C} > A r^{1-\beta_1} \frac{1}{C}  - r \frac{ 1}{C} > M r^a r^{1-\beta_1} M r^{c_2} - r M^{-1} r^{-c_1} =\\
= M^2 r^{1+a+c_2 - \beta_1} - M^{-1} r^{1-c_1}
\end{align*}
and for $p$ in $K_2$, we obtain:
\begin{align*}
\frac{A}{C r^{\beta_1-1}} + \frac{r B}{C} > A r^{1-\beta_1} \frac{1}{C}  - M r^{-b} \frac{ 1}{C} > 
M r^a r^{1-\beta_1} M r^{c_2} - M r^{-b} M^{-1} r^{-c_1} =\\
= M^2 r^{1+a+c_2 - \beta_1} - r^{1-c_1-b}
\end{align*}
So, it is clear that for $r$ sufficiently small and $\beta_1$ sufficiently big, the inequality $(\ref{eq:2})$ is satisfied for all points in $K = K_1 \cup K_2$, as we wanted to prove.
\end{proof}
\section{Globalization}
In this section, we prove that the existence of a polynomial $H$ smooth connected $\Gamma$-convergent implies the existence of a vector-field with an orbit $\gamma(t)$ such that its accumulation of secants set $Sec(\gamma(t))$ is $\alpha(\Gamma)$ (see the definition of the morphism $\alpha: \mathbb{R}^2 \longrightarrow \mathbb{S}^2$ in the introduction). More precisely, let $H$ be a smooth connected $\Gamma$-convergent and consider the subset of polynomials:
$$
\mathcal{O}^{'}_{\R^3} = \{ f \in \mathcal{O}_{\R^3}; \text{ } f(0,0,0)=0 \text{ and } f(x,y,0) \not\equiv 0\}
$$
and the function:
$$
\begin{array}{cccc}
\sigma: & \mathbb{R}^3 & \longrightarrow & \mathbb{R}^{3} \\
 & (x,y,z) & \mapsto & (xz,yz,z)
\end{array}
$$
which is also equal to a blowing-up restricted to one of its charts. Then, there exists two unique functions:
$$
\begin{array}{ccc}
\phi: \mathcal{O}^{'}_{\R^3} \longrightarrow \mathbb{N}^{\ast} & \text{  ,  } & \psi: \mathcal{O}_{\R^3} \longrightarrow \mathcal{O}^{'}_{\R^3}
\end{array}
$$
such that:
$$
z^{\phi(f)}f = \psi(f) \circ \sigma
$$
Now, fix constants:
\begin{itemize}
\item $\alpha = max\{\phi(H_y),\phi(H_x) \}$;
\item $\beta_1 \geq 2 [\phi(H_x)+\phi(H_y)]+\phi(H_z) - \alpha + 2$ (in particular $\beta_1 \geq 1$);
\item $\beta_2 = \phi(H)- \alpha +1$ (in particular $\beta_2 \geq 1$).
\end{itemize}
Notice that $\beta_1$ can be chosen as big as necessary. We define the vector-field:
$$
Y = A \frac{\partial}{\partial x} + B\frac{\partial}{\partial y} + C( x\frac{\partial}{\partial x} + y\frac{\partial}{\partial y} + z\frac{\partial}{\partial z})
$$
where:
\begin{itemize}
\item $A= - z^{a_1} \psi(H_y) + z^{a_2} \psi(H_x)\psi(H_z)$ where:
$$
\begin{array}{c}
a_1 = \alpha+1-\phi(H_y) \geq 1\\
a_2 = \alpha+\beta_1 + 1 - \phi(H_x) - \phi(H_z) \geq 1
\end{array}
$$
\item $B=  z^{b_1} \psi(H_x) + z^{b_2} \psi(H_y)\psi(H_z)$ where:
$$
\begin{array}{c}
b_1 = \alpha+1-\phi(H_x) \geq 1\\
b_2 = \alpha+\beta_1 + 1 - \phi(H_y) - \phi(H_z) \geq 1
\end{array}
$$
\item $C= - z^{c_1} \psi(H_x)^2  - z^{c_2} \psi(H_y)^2 +\psi(H)$ where:
$$
\begin{array}{c}
c_1 = \alpha+ \beta_1 -1 -2 \phi(H_x) \geq 1\\
c_2 = \alpha+ \beta_1 -1 -2 \phi(H_y) \geq 1
\end{array}
$$
\end{itemize}
The main result of this section can now be formulated precisely:
\begin{theorem}
Suppose that there exists a polynomial $H$ smooth connected $\Gamma$-convergent. Then, for $\beta_1$ sufficiently big, the vector-field $Y$ constructed above 
has a regular orbit $\gamma(t)$ such that the accumulation of secants $Sec(\gamma(t))$ is equal to $\alpha(\Gamma)$ (see the definition of the morphism $\alpha: \mathbb{R}^2 \longrightarrow \mathbb{S}^2$ in the introduction).
\label{prop:glo}
\end{theorem}
\begin{proof}
Consider the blowing-up of $\mathbb{R}^3$ with the origin as center:
$$
\tau : \mathbb{R} \times \mathbb{P}^2 \longrightarrow \mathbb{R}^3
$$
and let $Y^{'}$ be the strict transform of the vector-field $Y$. Then, in the $z$-chart, the strict transform $Y^{'}$ is equal to $X_{\beta_1,\beta_2}(H)$. Indeed:
$$
Y^{'} = z^{-\alpha} Y^{\ast} = z^{-(\alpha+1)} A^{\ast} \frac{\partial}{\partial x} + z^{-(\alpha+1)} B^{\ast} \frac{\partial}{\partial y} + z^{-\alpha}C^{\ast}( z \frac{\partial}{\partial z})
$$
where:
\begin{itemize}
\item $z^{-(\alpha+1)} A^{\ast} = - z^{a^{\ast}_1} H_y + z^{a^{\ast}_2 } H_x H_z$ where:
$$
\begin{array}{c} 
a^{\ast}_1 = \alpha+1-\phi(H_y) + \phi(H_y) - \alpha - 1 = 0\\
a^{\ast}_2 = \alpha+\beta_1 + 1 - \phi(H_x) - \phi(H_z) + \phi(H_x)+\phi(H_z)-\alpha-1 = \beta_1
\end{array}
$$
\item $z^{-(\alpha+1)} B^{\ast} = - z^{b^{\ast}_1} H_y + z^{b^{\ast}_2 } H_x H_z$ where:
$$
\begin{array}{c} 
b^{\ast}_1 = \alpha+1-\phi(H_x) + \phi(H_x) - \alpha - 1 = 0\\
b^{\ast}_2 = \alpha+\beta_1 + 1 - \phi(H_y) - \phi(H_z) + \phi(H_y)+\phi(H_z)-\alpha-1 = \beta_1
\end{array}
$$
\item $z^{-\alpha} C^{\ast} = - z^{c^{\ast}_1} H_x^2  - z^{c^{\ast}_2} H_y^2 +z^{c^{\ast}_3} H $ where: 
$$
\begin{array}{c}
c^{\ast}_1 = \alpha+ \beta_1 -1 -2 \phi(H_x) + 2 \phi(H_x) - \alpha = \beta_1-1\\
c^{\ast}_2 = \alpha+ \beta_1 -1 -2 \phi(H_y) + 2 \phi(H_y) - \alpha = \beta_1-1\\
c^{\ast}_3 = \phi(H)-\alpha = \beta_2-1
\end{array}
$$  
\end{itemize}
Since, by Proposition $\ref{prop:main}$, the vector-field $X_{\beta_1,\beta_2}(H)$ is $\Gamma$-convergent for $\beta_1$ big enough, the thesis clearly follows.
\end{proof}
\section{Isolated singularity}
\label{subsec:IS}
In this subsection, we prove that the vector-field $Y$ of Theorem $\ref{prop:glo}$ can be chosen so that the origin is an isolated singularity. For this end, we need to strengthen the hypotheses over the polynomial $H$ that is $\Gamma$-convergent.\\
\\
We say that a polynomial $H$ is \textit{isolated} $\Gamma$-convergent if:
\begin{itemize}
\item[i]) There exists a positive real number $\epsilon$ such that $L_{z}$ is a smooth variety for $-\epsilon<z<0$;
\item[ii]) The origin of $\mathbb{R}^2$ is an isolated solution of the equation in two variables $\{\psi(H)(x,y,0)=0\}$.
\end{itemize}
With this definition, we can prove the following result:
\begin{theorem}
Suppose that there exists a polynomial $H$ isolated smooth connected $\Gamma$-convergent. Then, the origin is an isolated singularity of the vector-field $Y$ constructed in Theorem $\ref{prop:glo}$.
\label{th:IS}
\end{theorem}
\begin{proof}
Notice that the vector-field $Y$ constructed in Theorem $\ref{prop:glo}$ is such that:
$$
Y(x,y,0) = C(x,y,0)( x\frac{\partial}{\partial x} + y\frac{\partial}{\partial y})
$$
and, by construction:
$$
C(x,y,0) = [- z^{c_1} \psi(H_x)^2  - z^{c_2} \psi(H_y)^2 +\psi(H)](x,y,0) = \psi(H)(x,y,0)
$$
because the natural constants $c_1$ and $c_2$ are not $0$. Since $H$ is isolated $\Gamma$-convergent, we conclude that there exists a positive real number $\delta$ such that the intersection of the singularities of $Y$ intersected with the variety $V(z)$, with the ball of radius $\delta$ and center in the origin is just the origin, i.e.:
$$
Sing(Y) \cap \{z=0\} \cap B_{\delta}(0) = \{(0,0,0)\}
$$
Thus, we only need to search singularities of $Y$ outside the plane $V(z)$, which can be done in the $z$-chart of the blowing-up of the origin. By construction, after blowing-up the origin the vector-field $Y$ is transformed into the vector-field $X_{\beta_1,\beta_2}(H)$. We now study the singularities of this vector-field:\\
\\
\textbf{Claim:} The set of singularities $Sing(X_{\beta_1,\beta_2})$ is equal to the variety $V(zH,H_x,H_y)$.
\begin{proof}
By construction of the vector-field $X_{\beta_1,\beta_2}(H)$, it is clear that $V(zH,H_x,H_y) \subset Sing(X_{\beta_1,\beta_2}) $. So consider a point $p$ in $Sing(X_{\beta_1,\beta_2}(H))$. We have that:
$$
\begin{array}{c}
\left[H_y - z^{\beta_1}H_z H_x\right](p) = 0 \\
\left[H_x + z^{\beta_1}H_z H_y \right](p)=0 \\
\left[-z^{\beta_1}(H^2_x+H^2_y) + z^{\beta_2}H\right](p)=0
\end{array}
$$
From the first equation, we have that $H_y(p) = z^{\beta_1}H_zH_x (p)$. Replacing this expression in the second equation, we get that $[H_x (1+ z^{2 \beta_1}H^2_z)] (p) = 0$, which implies that $H_x(p) = 0$. Thus, $H_y(p)=0$. Replacing this in the last equation, we finally get that either $z(p)=0$ or $H(p)=0$.
\end{proof}
Notice that the hypotheses over the polynomial $H$ implies that the varieties $L_z$ are all smooth for $z$ small enough (positive or negative) different from zero. This clearly implies that the intersection of the variety $V(H,H_x,H_y)$ with the set $\{-\epsilon<z<\epsilon\}$ is all contained in $\{z=0\}$. Furthermore, by the Claim, taking a smaller $\delta>0$ if necessary, the singularities of the vector-field $X_{\beta_1,\beta_2}(H)$ intersected with the set $\{-\epsilon<z<\epsilon\}$ is all contained in $\{z=0\}$ (which is the exceptional divisor). This finally implies that:
$$
Sing(Y) \cap B_{\delta}(0) = \{(0,0,0)\}
$$
and the origin is an isolated singularity.
\end{proof}
To finish the proof of Theorem $\ref{th:main}$, we need to prove the existence of polynomials $H$ isolated smooth connected $\Gamma$-convergent.
\begin{proposition}
If $\Gamma$ is equal to a semi-analytic variety $V(h) \cap \{x^2+y^2 \leq 1\}$, where $h$ is an adapted polynomial, then there exists a polynomial $H$ isolated smooth connected $\Gamma$-convergent.
\label{prop:IS}
\end{proposition}
%
%
\begin{proof}
We recall that:
$$
H_{\alpha}(x,y,z) = h(x,y)^2 + g(x,y,z)  + \alpha z^4
$$
where the parameter $\alpha$ takes values in the interval $[0,1]$ and the polynomial $g(x,y,z)$ has the form $z(\widetilde{g}(x,y,z) + \sum^N_{i=1} \binom{N}{i} x^{2i}y^{2(N-i)} )$ for some $N \in \mathbb{N}$, where $\widetilde{g}(x,y,z)$ is a polynomial with degree strictly smaller than $2N$.\\
\\
In particular, notice that $(0,0)$ is an isolated solution of the equation $\psi(g)(x,y,0)=0$. So, if we can choose $g$ such that $\phi(g) \geq \phi(h^2)$ and $\phi(g)>4$, we have that $(0,0)$ is an isolated solution of $\psi(H)(x,y,0)= 0$.\\
\\
But this property is not generally true for the polynomial $g(x,y,z)$ given in $\cite{Bel}$. So consider $f(x,y) = ((x-x_0)^{2} + (x-y_0)^{2})^M$ where $(x_0,y_0)$ is a point outside the ball with radius one and the variety $V(h)$, i.e. $p \notin V(h) \cup B_1(0)$, and $M$ is a natural number sufficiently big. It is clear that:
\begin{itemize}
\item The polynomial $ g(x,y,z) f(x,y)$ is equal to $z (\bar{g}(x,y,z) + \sum^{N+M}_{i=1} \binom{N+M}{i} x^{2i}y^{2(N-i)} )$, where $\bar{g}(x,y,z)$ is a polynomial with degree strictly smaller than $2(N+M)$;
\item For $M$ sufficiently big, we have that: $\phi(g f) \geq \phi(h^2)$ and $\phi(g f)>4$.
\end{itemize}
So, we consider:
$$
\widetilde{H}_{\alpha}(x,y,z) = h(x,y)^2 + f(x,y)g(x,y,z) -\alpha z^{4}
$$
It is now clear that property $[ii]$ of the definition of isolated $\Gamma$-convergent is satisfied. The proof that this polynomials (for almost all $\alpha$ in $[0,1]$) is isolated smooth connected $\Gamma$-convergent now follows, mutatis mutandis, the same proof of Propositions $[2.1-2.2]$ of $\cite{Bel}$. Nevertheless, we remark two crucial details for the necessary adaptations:
\begin{itemize}
\item Lemma $3.3$ of $\cite{Bel}$ contains the proof of property $[i]$ of the definition of isolated $\Gamma$-convergent, although it does not enunciate it;
\item The polynomial $f(x,y)$ is a positive unity over any small enough neighborhood of $B_1(0)$.
\end{itemize}
\end{proof}
\bibliographystyle{alpha}

\begin{thebibliography}{999999999}
\bibitem[Be]{Bel} Belotto, A. Analytic varieties as limit periodic sets, Qualitative Theory of Dynamical Systems October 2012, Volume 11, Issue 2, pp 449-465.
\bibitem[BiMi]{BM}  Bierstone, Edward; Milman, Pierre D. Semianalytic and subanalytic sets. Inst. Hautes Études Sci. Publ. Math. No. 67 (1988), 5?42.
\bibitem[CaMoR]{Fel2} F. Cano, R. Moussu, and J. P. Rolin, Non-oscillating integral curves and valuations. J. Reine Angew. Math. 582 (2005), 107-141.
\bibitem[CaMoS1]{Fel1} Cano, F.; Moussu, R.; Sanz, F. Nonoscillating projections for trajectories of vector fields. J. Dyn. Control Syst. 13 (2007), no. 2, 173-176.
\bibitem[CaMoS2]{Fel3} F. Cano, R. Moussu, and F. Sanz, Oscillation, spiralement, tourbillonnement. Comment. Math. Helv. 75 (2000), 284-318.
\bibitem[CaMoS3]{Fel4} F. Cano, R. Moussu, and F. Sanz, Pinceaux de courbes intégrales d'un champ de vecteurs analytique. Astérisque 297 (2004), 1-34.
\bibitem[Co]{Co} Coste, Michel, Ensembles semi-algebriques, \textit{Géométrie Algébrique Réelle et Formes Quadratiques}, Ed: Colliot-Thélène, J; E. Coste; M., E Mahé; L.; E Roy, M-F; Springer Berlin Heidelberg, 1982, 109-138.
\bibitem[GCaR]{Fel} C. Alonso-González, F. Cano, R. Rosas; Infinitesimal Poincaré-Bendixson Problem in dimension 3. arXiv:1212.2134 [math.DS]
\bibitem[Le]{M1} S. Lefschetz, Differential equations: geometric theory. Dover Publ., New York (1977).
\bibitem[Ly]{M2} A. M. Lyapunov, Stability of motion. Ann. Math. Stud. 17 (1947).
\bibitem[P]{M3} H. Poincaré, Mémoire sur les courbes définies par une équation différentielle. J. Math. 7 (1881).
\end{thebibliography}

\end{document}